\documentclass[11pt]{amsart}
\usepackage{}
\usepackage{amssymb}

\usepackage{color}

\theoremstyle{plain}
\newtheorem{Thm}{Theorem}

\newtheorem{Pro}[Thm]{Proposition}

\begin{document}

\title[scalar curvature of Yamabe solitons]
{Remarks on scalar curvature of Yamabe solitons}

\author{Li Ma, Vicente Miquel}

\address{Ma: Department of mathematics \\
Henan Normal university \\
Xinxiang, 453007 \\
China}

\email{lma@tsinghua.edu.cn}

\address{Vicente Miquel: Department of
Geometry and Topology, \\
The university  of Valencia\\
Av. Vicent Andr\'es EsteII\'es \\
 1 46100, Burjassot,
Valencia\\
Spain \\
}

\email{miquel@uv.es}

\thanks{The research is partially supported by the National Natural Science
Foundation of China 10631020 and SRFDP 20090002110019, the projects
DGI (Spain) and FEDER  project MTM2010-15444 and the Generalitat
Valenciana project GVPrometeo 2009/099}

\begin{abstract}
In this paper, we consider the scalar curvature of Yamabe solitons.
In particular we show that, with natural conditions and non positive Ricci curvature, any complete
Yamabe soliton  has constant scalar curvature,
namely, it is a Yamabe metric. We also show that the quadratic decay
at infinity of the Ricci curvature of a complete non-compact Yamabe
soliton has  non-negative scalar curvature. A new proof of
Kazdan-Warner condition is also presented.


\end{abstract}

 \maketitle

\section{Introduction}
In this work, we study the special solutions, the so called the
Yamabe solitons, to the Yamabe flow, which was introduced by
R.Hamilton at the same time as Ricci flow. We note that the Yamabe
flow has some similar properties as Ricci flow
(\cite{MC},\cite{Y94}, \cite{SS}\cite{H88},\cite{DM}). Since the
Yamabe solitons come naturally from the blow-up procedure along the
Yamabe flow \cite{AM}\cite{DM}\cite{chow}\cite{MC}, we are lead to
study the Yamabe solitons on complete non-compact Riemannian
manifolds. We shall study some properties of the scalar curvature of
the Yamabe solitons on complete non-compact Riemannian manifolds.
Recall that a Riemannian manifold $(M,g)$ is called a Yamabe soliton
if there are a smooth vector filed $X$ and constant $\rho$ such that
\begin{equation}\label{soliton}
(R-\rho)g=\frac{1}{2}L_Xg \ \ \text{on}  \ \ M,
\end{equation}
where $R$ is the scalar curvature and $L_Xg$ is the Lie derivative
of the metric $g$; When $X=\nabla f$ for some smooth function $f$,
we call it the gradient Yamabe soliton.The function $f$ above
will be called the potential function and it is determined up to a
constant. In this case the equation \eqref{soliton} becomes
\begin{equation}\label{solitong}
(R-\rho)g=\nabla^2 f \ \ \text{on}  \ \ M,
\end{equation}
When the constant $\rho\geq 0$, we call the Yamabe
solitons he non-expanding Yamabe solitons.

 In this paper we shall prove the following results on
the sign of the scalar curvature $R$ of a Yamabe soliton depending
on some asymptotic behaviour of it.

\begin{Thm}\label{max}
Let $(M,g)$ be a complete and non-compact gradient Yamabe soliton with
$\rho\geq 0$. Assume that $\underline{\lim}_{x\to\infty} R(x)\geq
0$. Then the scalar curvature $R$ of $(M,g)$ is non-negative.
Furthermore, if $(M,g)$ is not scalar flat, then $R>0$ on $M$.
\end{Thm}

We shall use the argument from \cite{P02} to get another results
about non-expanding Yamabe solitons.

\begin{Thm}\label{vicente+1}
Let $(M,g)$ be a complete and non-compact gradient Yamabe soliton with
$\rho\geq 0$. Assume that there is some point $x_0$ such that for
some large uniform  constant $R_0>1$,
$$
\int_{\gamma} [R- 2 (n-1) Ric(\gamma',\gamma')]\leq   \rho\ d(x),
$$
for any minimizing geodesic curve $\gamma$ connecting $x_0$ to $x$
with $d(x,x_0)\geq R_0$. Then $R\geq 0$.
\end{Thm}
The proof of this theorem will be given in section \ref{sect3}.

We can show that in some cases the Yamabe solitons are Yamabe
metrics.

\begin{Thm}\label{vicente+2}
Let $(M,g)$ be a complete and non-compact gradient Yamabe soliton such that
$|R-\rho|\in L^1(M)$, $\int_M Ric(\nabla f, \nabla f) \le 0$ and the potential function $f$ has at most
quadratic growth on $M$; that is,
$$
|f(x)|\leq Cd(x,x_0)^2, \ \ |\nabla f|\leq C(1+d(x,x_0)^2), \ \
$$
near infinity, where $C$ is some uniform constant
and $d(x,x_0)$ is
the distance function from the point $x$ to a fixed point $x_0$.
Then $R=\rho$ on $(M,g)$.
\end{Thm}

We shall also study Liouville type theorem of harmonic functions
with finite Dirichlet integral. We now show the following result.

\begin{Thm}\label{vicente+3}

Let $(M,g)$ be a complete and non-compact Riemmnian manifold with
non-negative Ricci curvature. Assume that $u$ is a harmonic function
with finite weighted Dirichlet integral, i.e., for some ball
$B(x_0)$,
$$
\int_{M-B(x_0)} d(x,x_0)^{-2}|\nabla u|^2<\infty.
$$
Then $\nabla^2u=0$ on $M$.
\end{Thm}

Then, use the idea of the proof of the above result to study the Yamabe solitons and we obtain

\begin{Thm}\label{thm4}
Assume that the Yamabe soliton $(M,g,X)$ has non-positive Ricci
curvature. Suppose that
\begin{equation}\label{hipt5}
\int_{M-B(x_0)} d(x,x_0)^{-2}
|X|^2<\infty.\end{equation}
 Then
$$
\nabla X=0
$$
and $R=\rho$.
\end{Thm}

Let us remark that Theorem 5 applies to solitons in general, they do not need to be gradient nor non-expanding. When applied to non-expanding solitons, Theorem \ref{thm4} states that the only non-expanding solitons with non positive Ricci curvature and satisfying condition  \eqref{hipt5} are the Ricci-flat and steady ones.

\section{Proofs of theorems 1, 4 and 5}

 \begin{proof} of Theorem \ref{max}.

We shall denote $Ric=(R_{ij})$ the Ricci tensor in
the local coordinates $(x^j)$.

First we shall obtain a formula for the laplacian os the scalar curvature of a gradient Yamabe soliton.
Taking the k-derivative to (\ref{solitong}) we have
$$
\nabla_kf_{ij}=\nabla_kRg_{ij}.
$$
Using the Ricci formula \cite{A98} we get that
$$
\nabla_if_{jk}+R_{jikl}f_l=\nabla_kRg_{ij}.
$$
By contraction for $j,k$,
$$
\nabla_i\Delta f+R_{il}f_l=\nabla_iR.
$$
Then we have
$$
nR_i+R_{il}f_l=R_i.
$$
Then we have $$ -R_{il}f_l=(n-1)R_l,
$$
or write in another way
\begin{equation}\label{riccf}
-Ric(\nabla f,\cdot)=(n-1)\nabla R.
\end{equation}

Taking one more derivative we have
$$
(n-1)\Delta R=-R_{il,i}f_l-R_{il}f_{il}.
$$
Recall the contracted Bianchi identity that
$$
R_{il,i}=\frac{1}{2}R_l.
$$
Then we have
$$
(n-1)\Delta R=-\frac{1}{2}(\nabla R,\nabla f)-R(R-\rho).
$$

we have
\begin{equation}\label{scalar}
(n-1)\Delta R+\frac{1}{2}g(\nabla f,\nabla R)+R^2-\rho R=0.
\end{equation}

Using the maximum principle we can prove Theorem 1.
 Assume that $\inf_M R(x)<0$. Since $\underline{\lim}_{x\to\infty} R(x)\geq
0$, we know that there is some point $z\in M$ such that $R(z)=\inf_M
R(x)<0$. Then we have $$ \Delta R(z)\geq 0, \ \ \nabla R(z)=0.
$$
By this we have at $z$ that
$$
(n-1)\Delta R+\frac{1}{2}g(\nabla f,\nabla R)\geq 0
$$
and by (\ref{scalar}),
$$ R(z)^2-\rho R(z)\leq 0.
$$
This is absurd since $R(z)^2-\rho R(z)>0$ for $\rho\geq 0$. The
strong maximum principle implies that either $R(x)>0$ or $R(x)=0$ on
$M$.
\end{proof}

\begin{proof} of Theorem \ref{vicente+3}

Recall the Bochner formula \cite{A98} that
\begin{equation}\label{bochner}
\frac{1}{2}\Delta |\nabla u|^2=|\nabla^2u|^2+g(\nabla u, \nabla \Delta u)+Ric(\nabla u,\nabla u).
\end{equation}
Then using the harmonicity of $u$, we have
$$
|\nabla^2u|^2+Ric(\nabla u,\nabla u)=\frac{1}{2}\Delta |\nabla u|^2.
$$
Choose a cut-off function $\phi=\phi_r$ on the ball $B_{2r}(x_0)$,
where $r>0$ (and we let $B_r=B_r(x_0)$ for simplicity) such that
$$
\phi_r=1, \ \ in \ \ B_r; \ \ |\nabla\phi_r|^2\leq \frac{C}{r^2},
$$
and
$$
\Delta \phi_r\leq \frac{C}{r^2}.
$$
These imply that
$$
\Delta \phi_r^2\leq \frac{C}{r^2}\to 0
$$
as $r\to \infty$. Then we have
$$
\int [|\nabla^2u|^2+Ric(\nabla u,\nabla u)]\phi_r^2=\int \frac{1}{2}\Delta |\nabla u|^2\phi_r^2.
$$

Using the integration by part and our assumption, we have
$$
\int \frac{1}{2}\Delta |\nabla u|^2\phi_r^2=\int \frac{1}{2}|\nabla
u|^2\Delta\phi_r^2,
$$
which is, by our assumption,
$$
\leq \int_{B_{2r}-B_r} \frac{C}{2r^2}|\nabla u|^2\to 0.
$$
as $r\to \infty$. Hence we have
$$
\int_M [|\nabla^2u|^2+Ric(\nabla u,\nabla u)]=0,
$$
which implies that $\nabla^2u=0$ and $Ric(\nabla u, \nabla u)=0$ on $M$.
\end{proof}

We now use the idea above to study the Yamabe solitons and give the

\begin{proof} of Theorem \ref{thm4}.
 By taking the trace, from the
defining equation of Yamabe soliton, we have that a Yamabe soliton satisfies
\begin{equation}\label{trace}
div X=n(R-\rho), \ \ \text{on} \ \ M.
\end{equation}
Recall the following Bochner formula \cite{Ya}
$$
div (L_Xg)(X)=\frac{1}{2}\Delta |X|^2-|\nabla X|^2+Ric(X,X)+\nabla_X
div(X).
$$
Then we have
\begin{equation}\label{bochner2}
|\nabla X|^2=\frac{1}{2}\Delta |X|^2+Ric(X,X)+(n-2)\nabla_X R.
\end{equation}
Fixing a cut-off function $\phi$ as above, we then have that
$$
\int X_j\nabla_jR\phi^2=-\int divX (R-\rho)\phi^2+2\phi\nabla_X\phi
(R-\rho).
$$
Hence
$$
\int \nabla_XR\phi^2=-n\int (R-\rho)^2\phi^2-2\int \phi \nabla_X\phi
(R-\rho).
$$

 Integrating (\ref{bochner2}) we have
$$
\int |\nabla X|^2\phi^2=\frac{1}{2}\int (\Delta\phi^2) |X|^2+\int
Ric(X,X)\phi^2+(n-2)\int\nabla_XR\phi^2.
$$
Then we obtain
$$
\int |\nabla X|^2\phi^2+n(n-2)\int (R-\rho)^2\phi^2=\frac{1}{2}\int
(\Delta\phi^2) |X|^2$$ $$+\int Ric(X,X)\phi^2-2(n-2)\int \phi
\nabla_X\phi (R-\rho).
$$
Using the Young and Cauchy-Schwartz inequalities we can get that
$$
\int |\nabla X|^2\phi^2+ (n-1)  (n-2)\int (R-\rho)^2\phi^2
$$
$$\leq
\frac{1}{2}\int (\Delta\phi^2) |X|^2+\int Ric(X,X)\phi^2+C(n)\int
|X|^2|\nabla\phi|^2
$$
for some uniform constant $C(n)$.

Then we have proved theorem \ref{thm4}.
\end{proof}

\section{proofs of theorem \ref{vicente+2} and related results}\label{sect2}

The argument of Theorem \ref{vicente+2} follows from the following proposition (see also \cite{DS}).

\begin{Pro}\label{mali} Let $(M,g)$ be a Yamabe solition with  smooth boundary. Then we have
$$
n (n-1) \int_M (R-\rho)^2  - \int_M Ric(\nabla f, \nabla f)  = (n-1) \int_{\partial M} (R-\rho) \nabla_\nu f.
$$
where $\nu$ is the outward unit normal to the boundary $\partial M$.
\end{Pro}

\begin{proof}
We use the argument from \cite{MD} (see also \cite{DS}. Note that
$$\int_M |\Delta f|^2=\int \Delta f f_{jj}.
$$
Integrating by parts we get that
$$
\int \Delta f f_{jj}=\int_{\partial M} \Delta f\nabla_\nu f-\int \nabla \Delta f\cdot \nabla f.
$$
Then using the Bochner formla \eqref{bochner} we have
$$
\int \Delta f f_{jj}=\int_{\partial M} {\color{red}n}\
(R-\rho)\nabla_\nu f+\int_M \left(|\nabla^2f|^2+Ric(\nabla f,\nabla
f)-\frac{1}{2}\Delta|\nabla f|^2 \right).
$$
Note that
\begin{align*}\int_M \Delta|\nabla f|^2 &= \int_{\partial M} \nabla_\nu |\nabla f|^2
= \int_{\partial M}2 \langle \nabla_\nu \nabla f, \nabla f\rangle \\
&= 2  \int_{\partial M} \nabla^2  f \langle\nu, \nabla f \rangle = 2 \int_{\partial M}(R-\rho) \langle \nu, \nabla f\rangle.
\end{align*}
Then we have
$$
\int_M |\Delta f|^2= \int_M \left( |\nabla^2f|^2 + Ric(\nabla f, \nabla f) \right) + (n-1) \int_{\partial M} { \color{red}}\  (R-\rho)\nabla_\nu f.
$$
And, using \eqref{solitong} in the above formula, we obtain
\begin{align*}
(n^2-n) \int_M (R-\rho)^2  - \int_M Ric(\nabla f, \nabla f)  = (n-1) \int_{\partial M} (R-\rho) \nabla_\nu f.
\end{align*}
\end{proof}

We now \emph{prove Theorem }\ref{vicente+2}.

\begin{proof}
By proposition \ref{mali}, we know that for the dimension constant $C_n>0$,
$$
C_n\int_{B_r}|R-\rho|^2 - \int_M Ric(\nabla f, \nabla f) = (n-1) \int_{\partial B_r}(R-\rho)\nabla_\nu f\leq
Cr\int_{\partial B_r}|R-\rho|.
$$
We now choose $r=r_j\to\infty$ such that
$$
r\int_{\partial B_r}|R-\rho|\to 0.
$$
This is obtained by using the fact that $\int_M|R-\rho|<\infty$ and
Fubini's theorem. Then, when $\int_M Ric(\nabla f, \nabla f) \le 0$, we have
$$
\int_{M}|R-\rho|^2=0,
$$
and this implies that $R=\rho$ on $M$.
\end{proof}

We take this chance to give another proof of Kazdan-Warner
condition; namely,

\begin{Pro}\label{kazdan-warner} Assume that $X$ is a conformal vector field on the compact Riemannian
manifold $(M,g)$, i.e., there exists a smooth function $a(x)$ on $M$
such that
$$
L_Xg=a(x)g.
$$
Then  we have
$$
\int_M\nabla_XRdv_g=-\frac{2n}{n-2}\int_{\partial M}
(Ric-\frac{R}{n}g)(\nu,X)d\sigma_g,
$$
where $\nu$ is the outer unit normal to the boundary $\partial M$.
\end{Pro}
\emph{Proof}. Set
$$
\dot{Ric}=Ric-\frac{R}{n}g.
$$
Then by the contracted Bianchi identity we get
$$
\delta\dot{Ric}=-\frac{n-2}{2n}dR.
$$
We now compute
$$
\int_M\nabla_XRdv_g=-\frac{2n}{n-2}\int_M \delta\dot {Ric}(X)dv_g.
$$
Integrating by part we get that
$$
\int_M \delta\dot {Ric}(X)dv_g=-\int_M (\dot {Ric},\nabla X)dv_g
+\int_{\partial M} \dot {Ric}(\nu,X)d\sigma_g.
$$

We then have
$$
\int_M \delta\dot {Ric}(X)dv_g=\int_{\partial M} \dot
{Ric}(\nu,X)d\sigma_g-\frac{1}{2}\int_M (\dot {Ric},L_Xg)dv_g.
$$
Recall that
$$
\frac{1}{2}L_Xg=\frac{1}{2} a(x)g.
$$
Since $(\dot {Ric},g)=0$, we obtain that
$$
\int_M \delta\dot {Ric}(X)dv_g=\int_{\partial M} \dot
{Ric}(\nu,X)d\sigma_g.
$$
This completes the proof of proposition \ref{kazdan-warner}.

\section{Proof of Theorem \ref{vicente+1}} \label{sect3}

The proof of Theorem \ref{vicente+1} will follow the argument of
pseudo-locality theorem due to Perelman \cite{P02}. The idea of
proof of Theorem \ref{vicente+1} is similar to Perelman's Li-Yau
harnack differential inequality. To make it, we recall some
well-known facts.

Define $d(x)=d(x,x_0)$. Let $\gamma(s)$( $\gamma:[0,\d(x)]\to M$) be
a shortest geodesic curve from $x_0$ to $x$. Without loss of
generality, we may assume that the distance function $d(x)$ is
smooth at $x$. Choose an orthonormal basis $(e_1, e_2,...,e_n)$ at
$x_0$ with $e_1=\gamma'(0)$. Extend the basis into a parallel basis
$(e_1(\gamma(s)), e_2(\gamma(s)),...,e_n(\gamma(s)))$ along the
curve $\gamma(s)$. Let $X_j(s)$ be the Jacobian vector field along
$\gamma(s)$ with $X_j(0)=0$ and $X_j(d(x))=e_j(d(x))$. Then we have
$$
\Delta d(x) =  \sum_j \int_0^{d(x)}(|X_j'(s)|^2 - R(\gamma',X_j,\gamma',X_j))ds.
$$
Fix some $r_0>0$ such that $|Ric|\leq (n-1)K$ on $B_{r_0}(x_0)$. Define
$$
Y_j(s)=a_j(s)e_j(s)
$$
for $j\geq 2$, where $a_j(s)$ is $\frac{s}{r_0}$ on $[0,r_0]$ and $a_j(s)=1$ on $[r_0,d(x)]$.

Using the minimizing property of the Jacobi field we have
$$
 \sum_j \int_0^{d(x)}(|X_j'(s)|^2-R(\gamma',X_j,\gamma',X_j))ds
 $$
 $$\leq \sum_j \int_0^{d(x)}(|Y_j'(s)|^2-R(\gamma',Y_j,\gamma',Y_j))ds.
$$
By direct computation \cite{P02} we have
\begin{align*}
 \sum_j &  \int_0^{d(x)}(|Y_j'(s)|^2-R(\gamma',Y_j,\gamma',Y_j))ds
\\
& = -\int_0^{d(x)}Ric(\gamma',\gamma')+\int_0^{r_0}(\frac{n-1}{r_0^2}+(1-\frac{s^2}{r_0^2})Ric(\gamma',\gamma'))ds
\end{align*}
and the latter is less than
$$
-\int_\gamma Ric(\gamma',\gamma')+(n-1)  \left(\frac{2}{3}Kr_0+ \frac1{r_0}\right).
$$
It is easy to see that
$$
g(\nabla f,\nabla d)=\nabla_{\gamma'} f(x)\leq \int_\gamma \nabla^2f(\gamma',\gamma')+|\nabla f(x_0)|.
$$
Using $$ \nabla^2f(\gamma',\gamma')=R-\rho
$$
we then have
$$
g(\nabla f,\nabla d)\leq -\rho d(x)+\int_\gamma R+|\nabla f(x_0)|.
$$
Hence, we have, for some uniform constant $C>0$,
\begin{equation}\label{star}
2 (n-1)\Delta d(x)+g(\nabla f,\nabla d)\leq-\rho d(x)+
\int_\gamma[-{\color{red} 2} \ (n-1)Ric(\gamma',\gamma')+R]+C/r_0^2.
\end{equation}
We may choose $r_0$ such that the latter is less than
$\frac{4(n-1)}{r_0^2}$.

For any fixed $A>2$ we shall consider the new function
$$
u(x)=\phi(\frac{d(x)}{Ar_0})R(x),
$$
where $\phi$ is   the cut-off function on the real line $\mathbb{R}$ defined after formula \eqref{bochner}, with $r=A r_0$. We
denote by $D=\frac{\phi^{''}}{\phi}$ and $h=\frac{\phi'}{\phi}$.

We compute
$$
\Delta u(x)=R\Delta \phi+2g(\nabla R,\nabla\phi)+\phi\Delta R.
$$

Note that $u=0$ outside the ball  of radius $2 A r_0$.

It is clear that if  $\inf_{M}u = 0$ for every $A$, then we have $R\geq 0$ on $M$.

If $\inf_{M}u< 0$ for some   $A=A_0$, then $\inf_{M}u< 0$ for every
$A>A_0$, and there is some point  $x_1\in B_{2 A r_0}(x_0)$  such
that
$$
u(x_1)=\inf_{M}u<0.
$$
Then we have $R(x_1)<0$. By this we have
$$
\phi'(x_1)R(x_1)>0,
$$
 which implies $x_1\notin  B_{A r_0}(x_0)$.  Moreover, at the minimum $x_1$,
\begin{equation}\label{comin}
\nabla u=0, \ \ \Delta u\geq 0
\end{equation}

The following
differential inequality by now is more or less  a standard computation
(see \cite{P02}),  but we shall give the details for the convenience of the reader. Using these two
properties \eqref{comin} and the equations (\ref{scalar}) and (\ref{star}) we can
get that
\begin{eqnarray*}
 \Delta u(x_1)&=&\left(\frac{D}{(Ar_0)^2}+\frac{h}{Ar_0}\Delta d\right)u(x_1)  + \frac1{2 (n-1)}   \frac{h}{Ar_0}(\nabla f,\nabla d)u(x_1)  \\
 &+&  \frac1{n-1} \rho\  u(x_1)-\phi R^2-2h^2\frac{1}{(Ar_0)^2}u(x_1)\\
&\leq& \left( \frac{D}{(Ar_0)^2}-\frac{ 2  h^2}{(Ar_0)^2}\right)u(x_1) - \frac1{n-1}   \phi R^2\\
&+& \frac{h}{Ar_0}[\Delta d  + \frac1{2 (n-1)}  \langle \nabla f,\nabla d\rangle]u(x_1)\\
&\leq & \left(
\frac{D}{(Ar_0)^2}-\frac{ 2 h^2}{(Ar_0)^2}\right)u(x_1)- \frac{1}{(n-1)\phi}  u(x_1)^2+\frac{2 \ h}{(Ar_0)^2}  u(x_1).
\end{eqnarray*}

Then $$ \Delta u(x_1)\leq
\frac{|u(x_1)|}{\phi} \left\{\frac{1}{A^2 r_0^2} \left[\frac{ 2 \phi'^2}{\phi}+ 2 |\phi'| + |\phi^{''}| \right]- \frac1{n-1} |u(x_1)| \right\}.
$$
 For some uniform constant $C>0$, we have
$$
 2 |\phi'|\leq C,\ \  \frac{2 \phi'^2}{\phi}\leq C, \ \ |\phi^{''}|\leq
C,
$$
Then we can show that
$$
|u(x_1)|\leq   \frac{(n-1)\ C}{A^2 r_0^2}.
$$
The latter implies that
$$
R(x)\geq  -  \frac{(n-1)\ C}{A^2 r_0^2} \  \ \ on \ \  B_{{2}Ar_0}(x_0) .
$$

Sending $A\to \infty$, we get that $R\geq 0$ on $M$.

This completes the proof of Theorem \ref{vicente+1}.

\emph{Acknowledgement.} Part of this work was done when the first
named author was visiting Valencia University and he would like to
thank the hospitality of its Department of Geometry and Topology.

\end{document}